\documentclass{article}
\usepackage{latexsym,amsthm,amssymb,amscd,amsmath}
\newtheorem{thm}{Theorem}[section]

\newtheorem{prop}[thm]{Proposition}

\newtheorem{lem}[thm]{Lemma}

\newcommand{\be}{\begin{equation}}
\newcommand{\ee}{\end{equation}}
\newcommand{\ben}{\begin{enumerate}}
\newcommand{\een}{\end{enumerate}}
\newcommand{\beq}{\begin{eqnarray}}
\newcommand{\eeq}{\end{eqnarray}}
\newcommand{\beqn}{\begin{eqnarray*}}
\newcommand{\eeqn}{\end{eqnarray*}}

\newcommand{\pa}{\partial}

\title{On Kropina Change of $m$-th Root Finsler Metrics}

\author{A. Tayebi, T. Tabatabaeifar and  E. Peyghan}
\begin{document}

\maketitle
\begin{abstract}
In this paper, we consider Kropina change of $m$-th  root Finsler metrics. We find necessary and sufficient condition under which the Kropina change of an $m$-th  root Finsler metric be locally dually flat. Then we prove that the Kropina change of an $m$-th root Finsler metric is locally projectively flat if and only if it is locally Minkowskian.
\end{abstract}

\section{Introduction}
Let $M$ be an $n$-dimensional $C^{\infty}$ manifold, $TM$ its tangent bundle. Let $F=\sqrt[m]{A}$ be a Finsler metric on $M$, where $A$ is given by $A:=a_{i_{1}\dots i_{m}}(x)y^{i_{1}}y^{i_{2}}\dots y^{i_{m}}$ with $a_{i_{1}\dots i_{m}}$ symmetric in all its indices \cite{Mangalia}\cite{Shim}\cite{TN1}\cite{TN2}\cite{TPShah}. Then $F$ is called an $m$-th root Finsler metric.  Suppose that $A_{ij}$ define a positive definite tensor and $A^{ij}$ denotes its inverse. For an $m$-th root metric  $F$, put
\[
A_{i}={\pa A\over \pa y^i}, \ \  A_{ij}={\pa^2 A\over \pa y^j\pa y^j}, \ \  A_{x^i}=\frac{\partial A}{\partial x^i}, \ \  A_0=A_{x^i}y^i.
\]
Then the following hold
\begin{eqnarray}
&&g_{ij}=\frac{A^{\frac{2}{m}-2}}{m^2}[mAA_{ij}+(2-m)A_iA_j],\label{gg}\\
&&y^iA_i=mA, \ \ y^iA_{ij}=(m-1)A_j,\ \ y_i=\frac{1}{m}A^{\frac{2}{m}-1}A_i,\\
&&A^{ij}A_{jk}=\delta^i_k,\ \ A^{ij}A_i=\frac{1}{m-1}y^j, \ \ A_iA_jA^{ij}=\frac{m}{m-1}A.
\end{eqnarray}
Let $(M,F)$ be a Finsler manifold. For a $1$-form $\beta(x,y)=b_i(x)y^i$ on $M$, we have a change of Finsler which is defined by following
\be
F(x,y) \rightarrow \bar F(x,y)=f(F,\beta),
\ee
where $f(F,\beta)$ is a positively homogeneous function of $F$. This is called a $\beta$-change of metric. It is easy to see that,  if $||\beta||_F:=\sup_{F(x,y)=1}|b_i(x)y^i|<1$, then $\bar F$ is again a Finsler metric \cite{Shib}.

In this paper, we consider a special case of $\beta$-change, namely
\be
\bar F(x,y)=\frac{F^2(x,y)}{\beta(x,y)}\label{KC}
\ee
which is called the Kropina change. If $F$ reduces to a  Riemannian metric $\alpha$, then ${\bar F}$ reduces to a Kropina metric $F=\frac{\alpha^2}{\beta}$.  Due to this reason, the transformation (\ref{KC}) has been called the Kropina change of Finsler metrics. It is remarkable that, the Kropina metrics are closely related to physical theories. These metrics, was introduced by Berwald in connection with a two-dimensional Finsler space with rectilinear extremal  and was investigated by  Kropina \cite{Mat5}.

In \cite{amna}, Amari-Nagaoka introduced the notion of dually flat Riemannian metrics when they study the information geometry on Riemannian manifolds. Information geometry has emerged from investigating the geometrical structure of a family of probability distributions and has been applied successfully to various areas including statistical inference, control system theory and multi-terminal information theory \cite{am}. In Finsler geometry, Shen extends the notion of locally dually flatness for Finsler metrics \cite{shen}.   A Finsler metric $F$ on an open subset $U\subset \mathbb{R}^n$ is called dually flat if it satisfies $(F^2)_{x^ky^l}y^k=2(F^2)_{x^l}$.

In this paper, we  find necessary and sufficient condition under which a Kropina change of an $m$-th  root metric be locally dually flat.
\begin{thm}\label{mainthm1}
Let  $F=\sqrt[m]{A}$ be an $m$-th root Finsler metric on an open subset $U\subset \mathbb{R}^n$, where $A$ is irreducible. Suppose that $\bar{F}=\frac{F^2}{\beta}$ be Kropina change of $F$ where $\beta=b_i(x)y^i$. Then $\bar{F}$ is locally dually flat if and only if there exists a 1-form $\theta = \theta_{l} (x)y^l$  on U such that
the following hold
\begin{eqnarray}
&&\beta_{0l}\beta -3\beta_l\beta_0=2\beta\beta _{x^{l}},\label{SRR1}\\
&&A_{x^l}=\frac{1}{3m}[mA\theta_{l}+4\theta A_{l}],\label{SRR2}\\
&&\beta_0A_l=-\beta_lA_0,
\end{eqnarray}
where $\beta _{0l}=\beta _{x^k y^l}y^k$, $\beta_{x^l}=(b_i)_{x^l}y^i$, $\beta_0=\beta_{x^l}y^i$ and $\beta_{0l}=(b_l)_0$.
\end{thm}
A Finsler metric is said to be  locally projectively flat if at any point there is a local coordinate system  in which the geodesics are straight lines as point sets. It is known that a Finsler metric  $F(x,y)$ on an open domain $ U\subset \mathbb{R}^n$   is  locally projectively flat  if and only if $G^i= Py^i$, where $P(x,  \lambda y) = \lambda P(x, y)$, $\lambda >0$ \cite{shLi1}.

In this paper, we prove that the Kropina change of an $m$-th root Finsler metric is locally projectively flat if and only if it is locally Minkowskian.

\begin{thm}\label{mainthm2}
Let  $F=\sqrt[m]{A}$ be an $m$-th root Finsler metric on an open subset $U\subset \mathbb{R}^n$, where $A$ is irreducible. Suppose that $\bar{F}=\frac{F^2}{\beta}$ be Kropina change of $F$ where $\beta=b_i(x)y^i$. Then  $\bar{F}$ is locally projectively flat if and only if it is locally Minkowskian.
\end{thm}

\section{ Proof of the Theorem \ref{mainthm1}}
A Finsler metric $F=F(x,y)$ on a manifold $M$ is said to be locally dually flat if at any point there
is a standard coordinate system  $(x^i,y^i)$ in $TM$ such that $L=F^2(x,y)$ satisfies
\be
L_{x^ky^l}y^k=2L_{x^l}.\label{1}
\ee
In this case, the coordinate $(x^i)$ is called an adapted local coordinate system. It is easy to see that every locally Minkowskian metric satisfies in the above equation, hence is locally dually flat \cite{TPS1}\cite{TPS2}.

In this section, we are going to prove the Theorem \ref{mainthm1}. To prove it, we need the following.

\begin{lem}\label{lemp}
Suppose that the equation $\Phi A^2+\Psi A+\Theta=0$ holds,  where $\Phi, \Psi,\Theta $  are polynomials in $y$ and $m>2$. Then $\Phi=\Psi=\Theta=0$.
\end{lem}

\bigskip

\noindent {\it\bf Proof of Theorem \ref{mainthm1}}: Let $\bar{F}$ be a locally dually flat metric. We have
\begin{eqnarray*}
{\bar F}^2 \!\!\!\!&=&\!\!\!\!\!\frac{A^{\frac{4}{m}}}{\beta^2},
\\
({\bar F}^2)_{x^k}\!\!\!\!&=&\!\!\!\! \frac{1}{\beta^2}\frac{4}{m}A^{\frac{4}{m}-1}A_{x^k}-\frac{2}{\beta^3}A^{\frac{4}{m}}{\beta}_{x^k},
\\
({\bar F}^2)_{x^{k}y^{l}}y^k\!\!\!\!&=&\!\!\!\! \frac{1}{\beta^2}
\big[\frac{4}{m}A^{\frac{4}{m}-1}A_{0l}+(\frac{4}{m})(\frac{4}{m}-1)A^{\frac{4}{m}-2}A_0A_l\big]\\
\!\!\!\!&-&\!\!\!\!\frac{2}{\beta^3}\big[\frac{4}{m}A^{\frac{4}{m}-1}A_l\beta_0+\frac{4}{m}A^{\frac{4}{m}-1}A_0\beta_l
+A^{\frac{4}{m}}\beta_{0l}+
\frac{6}{\beta^4}\big[A^{\frac{4}{m}}\beta_0\beta_l \big]
\end{eqnarray*}
Thus, we get
\begin{eqnarray}
\nonumber\frac{A^{\frac{4}{m}-2}}{\beta^4}\!\!\!&\bigg[&\!\!\!\frac{4}{m}\beta^2\big[(\frac{4}{m}-1)A_0A_l
+AA_{0l}-2AA_{x^l}\big]-\frac{8}{m}A\beta[A_l\beta_0+A_0\beta_l]\\
\!\!\!&+&\!\!\!2A^2\big[3\beta_0\beta_l+2\beta\beta_{x^l}-\beta\beta_{0l}]\bigg]=0
\end{eqnarray}
By Lemma \ref{lemp},  we have
\begin{eqnarray}
&&(\frac{4}{m}-1)A_lA_0+AA_{0l}=2AA_{x^{l}},\label{m6}\\
&&\beta_0A_l=-A_0\beta_l,\\
&&\beta _{0l}\beta -3\beta _{l}\beta _{0}=2\beta_{x^{l}}\beta,
\end{eqnarray}
One can rewrite (\ref{m6}) as follows
\be
A(2A_{x^l}-A_{0l})=(\frac{4}{m}-1)A_lA_0.\label{d11}
\ee
Irreducibility of $A$ and $deg(A_l)=m-1$ imply that there exists a 1-form $\theta=\theta_l y^l$ on $U$ such that
\be
A_0=\theta A.\label{d12}
\ee
Plugging (\ref{d12}) into (\ref{d11}), yields
\be
A_{0l}=A\theta_l+\theta A_l-A_{x^l}.\label{d13}
\ee
Substituting (\ref{d12}) and (\ref{d13}) into (\ref{d11}) yields (\ref{SRR2}).  The converse is a direct computation. This completes the proof.
\qed

\section{ Proof of the Theorem \ref{mainthm2}}
A Finsler metric $F(x,y)$ on an open domain $ U\subset \mathbb{R}^n$ is said to be  locally projectively flat if  its geodesic coefficients $G^i$ are in the form
\[
G^i(x, y) = P(x, y) y^i,
\]
where $P: TU = U\times \mathbb{R}^n \to \mathbb{R}$ is positively homogeneous with degree one, $P(x, \lambda y) = \lambda P(x, y)$, $\lambda >0$. We call $P(x, y)$ the  projective factor of $F$.

In this section, we are going to prove the Theorem \ref{mainthm2}. To prove it,  we need the following.

\begin{prop}\label{lemb}
Let  $F=A^{\frac{1}{m}}$ be an $m$-th root Finsler metric on an open subset $U\subset \mathbb{R}^n$ ($n\geq 3$), where $A$ is irreducible. Suppose that $\bar{F}=\frac{F^2}{\beta}$ be Kropina change of $F$ where $\beta=b_i(x)y^i$.  If $\bar{F}$ is projectively flat metric then it reduces to a Berwald metric.
\end{prop}
\begin{proof}
Let $\bar{F}$ be projectively flat metric. We have
\begin{eqnarray*}
\bar {F}_{x^k}\!\!\!&=&\!\!\! \frac{2}{m\beta}A^{\frac{2}{m}-1}A_{x^k}-\frac{1}{\beta^2}A^{\frac{2}{m}}{\beta}_{x^k},
\\
{\bar F}_{x^{k}y^{l}}y^k\!\!\!&=&\!\!\!\frac{1}{\beta}
\big[\frac{2}{m}A^{\frac{2}{m}-1}A_{0l}+(\frac{2}{m})(\frac{2}{m}-1)A^{\frac{2}{m}-2}A_0A_l\big]\\
\!\!\!&-&\!\!\!\frac{1}{\beta^2}\big[\frac{2}{m}A^{\frac{2}{m}-1}A_l\beta_0+\frac{2}{m}A^{\frac{2}{m}-1}A_0\beta_l+A^{\frac{2}{m}}\beta_{0l}+
\frac{2}{\beta^3}\big[A^{\frac{2}{m}}\beta_0\beta_l \big]
\end{eqnarray*}
Thus, we get
\begin{eqnarray*}
\frac{A^{\frac{2}{m}-2}}{\beta^3}\!\!\!&\bigg[&\!\!\!\frac{2}{m}\beta^2\big[(\frac{2}{m}-1)A_0A_l+AA_{0l}-AA_{x^l}\big]\\
\!\!\!&-&\!\!\!\frac{2}{m}A\beta[A_l\beta_0+A_0\beta_l]
+A^2\big[2\beta_0\beta_l+\beta\beta_{x^l}-\beta\beta_{0l}]\bigg]=0
\end{eqnarray*}
By Lemma \ref{lemp},  we have
\begin{eqnarray}
mA(A_{0l}-A_{x^l})=(m-2)A_0A_l.
\end{eqnarray}
Then irreducibility  of $A$ and $deg(A_l )= m-1<deg(A)$ implies that $A_0$ is divisible by $A$. This means that, there is a 1-form $\theta=\theta_l y^l$ on $U$
such that the following holds
\[
A_0=2m A\theta.
\]
Then $G^i = Py^i $, where $P=\theta$. Then F is a Berwald metric.
\end{proof}

\bigskip

\noindent {\it\bf Proof of Theorem \ref{mainthm2}}: By Proposition \ref{lemb},  if $F$
is projectively flat then it reduces to a Berwald metric. Now, if $n\geq 3$ then by Numata's theorem every Berwald metric of non-zero scalar flag curvature ${\bf K}$ must be Riemaniann. This is contradicts with our assumption. Then ${\bf K}=0$, and in this case $F$ reduces to a locally Minkowskian metric.
\qed

\noindent
Akbar Tayebi and Tayebeh Tabatabaeifar\\
Department of Mathematics, Faculty  of Science\\
University of Qom \\
Qom. Iran\\
Email:\ akbar.tayebi@gmail.com\\
Email:\ t.tabaee@gmail.com
\bigskip

\noindent
Esmaeil Peyghan\\
Department of Mathematics, Faculty  of Science\\
Arak University\\
Arak 38156-8-8349,  Iran\\
Email: epeyghan@gmail.com

\end{document}